\documentclass{amsart}

\usepackage{amsmath,amssymb,amsthm,url,amsfonts}   
\usepackage{comment}
\usepackage{mathrsfs}  
\usepackage{mathtools}
\usepackage{bbm}
\usepackage{mathabx}    
\usepackage{color}
\usepackage{enumerate}  
\usepackage[shortlabels]{enumitem}

{\theoremstyle{plain}
	\newtheorem{definition}{Definition}[section]
	\newtheorem{lemma}[definition]{Lemma}
	\newtheorem{theorem}[definition]{Theorem}
	\newtheorem{corollary}[definition]{Corollary}

}
    {\theoremstyle{definition}
	\newtheorem{example}[definition]{Example}
	\newtheorem{remark}[definition]{Remark}
}
    \numberwithin{equation}{section}

\renewcommand{\mathbb}{\mathbbm}                     
\renewcommand{\epsilon}{\varepsilon}                 
\renewcommand{\phi}{\varphi}
\renewcommand{\le}{\leqslant}
\renewcommand{\ge}{\geqslant}
\newcommand{\origsetminus}{} \let\origsetminus=\setminus  
\renewcommand{\setminus}{\!\origsetminus\!}
\newcommand{\origfoo}{} \let\origfoo=\sqrt           
\renewcommand{\sqrt}[1]{\origfoo{#1}\;}

\newcommand{\abs}[1]{\left\lvert #1 \right\rvert}    
\newcommand{\norm}[1]{\left\lVert #1 \right\rVert}   
\newcommand{\scapro}[2]{\left\langle #1,#2 \right\rangle}       

\DeclareMathOperator{\1}{\mathbbm 1}
\DeclareMathOperator{\Id}{I}                       

\newcommand{\fa}{\qquad\text{for all }}
\newcommand{\leb}{\rm{leb}}

\newcommand{\n}{\Vert}
\newcommand{\om}{\omega}
\newcommand{\Om}{\Omega}
\newcommand{\PP}{\mathbb{P}}
\newcommand{\EE}{\mathbb{E}}
\newcommand{\cE}{\mathscr{E}}
\newcommand{\cF}{\mathscr{F}}
\newcommand{\la}{\lambda}

\newcommand{\calF}{\mathscr{F}}
\newcommand{\calG}{\mathscr{G}}

\DeclareMathOperator{\Borel}{{\mathfrak B}}

\newcommand{\D}{{\mathcal D}}

\renewcommand{\L}{{\mathcal L}}

\renewcommand{\S}{{\mathcal S}}

\newcommand{\I}{{\mathcal I}}
\renewcommand{\d}{{\mathrm d}}
\newcommand{\cp}{\widetilde{N}}

\DeclareMathOperator{\R}{{\mathbb R}}                
\DeclareMathOperator{\Rp}{{\mathbb R_+}}             
\DeclareMathOperator{\C}{{\mathbb C}}                
\DeclareMathOperator{\N}{{\mathbb N}}                

\DeclareMathOperator{\E}{{\mathbb E}}

\newcommand{\jan}[1]{{\color{red}{#1} }}

\allowdisplaybreaks

\title{L\'evy measures on Banach spaces}

\author{Jan van Neerven}

\address{
Delft Institute of Applied Mathematics,
Delft University of Technology,
P.O. Box 5031,
2600 GA Delft,
The Netherlands
        }

\email{J.M.A.M.vanNeerven@TUDelft.nl}

\author{Markus Riedle}

\address{
Department of Mathematics,
King's College London,
Strand,
London WC2R 2LS,
UK
        }

\email{markus.riedle@kcl.ac.uk}



\begin{document}

\begin{abstract}
In this work, we establish an explicit characterisation of L\'evy measures on both $L^p$-spaces and UMD Banach spaces. In the case of $L^p$-spaces, L\'evy measures are characterised by an integrability condition, which directly generalises the known description of L\'evy measures on sequence spaces. The latter has been  the only known description of L\'evy measures on infinite dimensional Banach spaces that are not Hilbert. L\'evy measures on UMD Banach spaces are characterised by the finiteness of the expectation of a random $\gamma$-radonifying norm. Although this description is more abstract, it reduces to simple integrability conditions in the case of $L^p$-spaces.
\end{abstract}

\maketitle

\section{Introduction}

A $\sigma$-finite measure $\lambda$ on the Borel $\sigma$-algebra $\Borel(U)$ over a separable Banach space $(U, \Vert\cdot\Vert)$ with $\lambda(\{0\})=0$ is called a L\'evy measure if the function $\phi_\varrho\colon U^\ast\to \C$, defined by
\begin{align}\label{eq.char}
	\phi_\varrho(u^\ast)=\exp\left(\int_{U}  \left(e^{i\scapro{u}{u^\ast}}-1-i\scapro{u}{u^\ast}\1_{B_U}(u)\right)\, \lambda(\d u)\right), \qquad u^\ast\in U^\ast,
\end{align}
is the characteristic function of a probability measure $\varrho$ on $\Borel(U)$. Here, $B_U$ is the closed unit ball of $U$, and $U^\ast$ denotes the Banach space dual of $U$.
here and in the rest of the paper $\scapro{u}{u^\ast}$ denotes the value of the functional $u^\ast\in U^\ast$ on the element $u\in U$. General references to the theory of L\'evy measures include the books by Linde \cite{Linde} and Sato \cite{Sato}.

In the case where $U$ is finite-dimensional, it is well known that a $\sigma$-finite measure $\lambda$ on the Borel $\sigma$-algebra $\Borel(\R^d)$  with $\lambda(\{0\})=0$ is a L\'evy measure if and only if
\begin{align}\label{eq.int-R}
	\int_{\R^d}(|r|^2\wedge 1)\, \lambda (\d r)<\infty.
\end{align}
Indeed, the latter often serves as the definition of a L\'evy measure on $\R^d$ in the literature.

Replacing the Euclidean norm $\abs{\cdot}$ by the Hilbert space norm, the above equivalent characterisation of L\'evy measures extends  to separable Hilbert spaces; see Parthasarathy~\cite{Parthasarathy}.

Surprisingly, although the integrability condition \eqref{eq.int-R} can be formulated in Banach spaces,
this characterisation of L\'evy measures ceases to hold in arbitrary Banach spaces $U$.
In fact, for the case $U=C[0,1]$, the space of continuous functions on $[0,1]$ endowed with the supremum norm, it is shown in Araujo~\cite{Araujo} that
there exists a $\sigma$-finite Borel measure $\lambda$ on $\Borel(U)$ with $\lambda(\{0\})=0$ and satisfying
\begin{align}\label{eq.int}  \int_{U}(\norm{u}^2\wedge 1)\, \lambda (\d u)<\infty,
\end{align}
but the function $\phi_\varrho$ defined in \eqref{eq.char}
is not the characteristic function of a Borel measure $\varrho$ on $\Borel(U)$. Vice versa, there exists a $\sigma$-finite Borel measure $\lambda$  on $\Borel(U)$ with $\lambda(\{0\})=0$ such that \eqref{eq.char} is the characteristic function of a probability Borel measure on $\Borel(U)$ but $\lambda$ does not satisfy \eqref{eq.int}.

Explicit characterisations of L\'evy measures on an infinite dimensional Banach space, which is not a Hilbert space, is known for the spaces $\ell^p(\N)$ of summable sequences for $p\ge 1$, and for $L^p$ with $p\ge 2$. The former result was derived in Yurinskii \cite{Yurinskii} by means of a two-sided $L^p$-bound of compensated Poisson measure (for which he credited Novikov, who is also credited in Marinelli and R\"ockner \cite{MarinelliRockner-purely} for similar estimates). The latter is due to \cite{GMZ79}.

A sufficient condition in terms of an integrability condition similar to \eqref{eq.int} is known in Banach spaces $U$ of Rademacher type $p\in [1,2)$. In the converse direction, in Banach spaces $U$ of Rademacher cotype $q\in [2,\infty)$ it is known for a $\sigma$-finite measure $\lambda$, that if \eqref{eq.char} defines the characteristic function of a probability measure on $U$ then $\lambda$ satisfies an integrability condition similar to \eqref{eq.int}. In fact, these necessary or sufficient conditions can be used to characterise Banach spaces of Rademacher type $p\in [1,2]$ or of Rademacher cotype $q\in [2,\infty)$. These results can be found in Araujo and Gin\'e \cite{AraujoGine2}.

In the present paper, we derive explicit characterisations of L\'evy measures for both $L^p$-spaces and UMD Banach spaces. In the case of $L^p$-spaces, L\'evy measures are characterised by an integrability condition, which directly generalises the aforementioned results for $\ell^p(\N)$ for $p\ge 2$ by Yurinskii \cite{Yurinskii}. L\'evy measures on UMD Banach spaces are characterised by the finiteness of the expectation of a random $\gamma$-radonifying norm. Although the latter description is more abstract, we demonstrate its applicability by deducing similar integrability conditions for the special cases of $L^p$-spaces as obtained earlier by different arguments.

For both $L^p$-spaces and UMD spaces, our method relies on recently achieved two-sided $L^p$-estimates of integrals of vector-valued deterministic functions with respect to a compensated Poisson random measure in Dirksen \cite{Dirksen} and Yaroslavtsev \cite{Yaroslavtsev-20}. Such inequalities are sometimes called Bichteler-Jacod or Kunita inequalities and suggested to be called Novikov inequalities in \cite{MarinelliRockner-purely}, where more historical details can be found.  Since the results in Dirksen \cite{Dirksen} and Yaroslavtsev \cite{Yaroslavtsev-20} are only formulated for simple functions, we provide the straightforward arguments for their extension to arbitrary vector-valued deterministic functions.

Throughout the paper, all vector spaces are real. We write $\R_+ = (0,\infty)$ and $\N:=\{1,2,3,\dots\}$. We use the shorthand notation $$A\simeq_q B$$ to express that the two-sided inequality $c_q A \le B \le c_q'A$ holds with constants $0<c_q\le c_q'<\infty$ depending only on $q$.

\section{Preliminaries}\label{se.PRM}

Throughout this paper, we let $U$ be a {\em separable} Banach space with dual space $U^\ast$ and duality pairing $\scapro{\cdot}{\cdot}$.
The Borel $\sigma$-algebra on $U$ is denoted by $\Borel(U)$.
By a standard result in measure theory (e.g., \cite[Proposition E.21]{Nee}) the separability of $U$ implies that every finite Borel measure $\varrho$ on $U$ is a {\em Radon measure}, that is, that for all Borel sets $B\in \Borel(U)$ and $\varepsilon>0$ there exists a compact set $K$ such that $K\subseteq B$ and $\varrho(B\, \setminus \, K)<\varepsilon$. Such measures are uniquely described by their {\em characteristic function}, which is the function $\phi_\varrho\colon U^\ast\to\C$ given by
$$\qquad \phi_\varrho(u^\ast)=\int_U e^{i\scapro{u}{u^\ast}}\, \varrho(\d u).$$

\subsection{L\'evy measures}\label{ss.PRM-Levy}

In what follows we write
$$B_{U,r} := \{u\in U:\, \norm{u}\le r\}$$ for the closed ball in $U$ of radius $r>0$ centred at the origin. Its complement is denoted by $B_{U,r}^c$. We furthermore write $B_{U}:= B_{U,1}$ and $B_{U}^c:= B_{U,1}^c$ for the closed unit ball of $U$ and its complement.

A $\sigma$-finite measure $\lambda$ on the Borel $\sigma$-algebra $\Borel(U)$ with $\lambda(\{0\})=0$ is called a {\em L\'evy measure} if the function $\phi\colon U^\ast\to \C$ defined by
\begin{align}\label{eq.char-lambda}
	\phi(u^\ast)=\exp\left(\int_U \left(e^{i\scapro{u}{u^\ast}}-1-i \scapro{u}{u^\ast}\1_{B_{U}}(u)\right)\, \lambda(\d u)\right)
\end{align}
is the characteristic function  of a probability measure $\eta(\lambda)$ on $\Borel(U)$. For any $r>0$, we often decompose
\begin{align}\label{eq:decompose-lambda} \lambda=\lambda|_r+\lambda|_r^c,
\end{align}
where
$$\lambda_r(\cdot):=\lambda(\cdot\cap B_{U,r}) \ \ \hbox{ and } \ \  \lambda|_r^c(\cdot):=\lambda(\cdot \cap B_{U,r}^c).$$

Every {\em finite} measure $\lambda$ on $\Borel(U)$ with $\lambda(\{0\})=0$ is a L\'evy measure. In this case, a 
probability measure $\pi(\lambda)$ on $\Borel(U)$ is defined by 
\[ \pi(\lambda)(B)=e^{-\lambda(U)} \sum_{k=0}^\infty \frac{\lambda^{\ast k}(B)}{k!}. \]
We further define
\[ \eta(\lambda):= \pi(\lambda)\ast \delta_{s(\lambda)}, \qquad\text{where } s(\lambda):=-\int_{B_U} u\, \lambda(\d u),\]
has characteristic function given by \eqref{eq.char-lambda}.
Here, $\ast$ denotes convolution, and $\lambda^{\ast k}$ denotes the $k$-fold convolution of $\lambda$ with itself.

\begin{theorem}[\hbox{\cite[Theorem 5.4.8]{Linde}}]\label{th.Levy-equivalence}
	Let $\lambda$ be a $\sigma$-finite measure measure on $\Borel(U)$ satisfying $\lambda(\{0\})=0$. The following assertions are equivalent:
	\begin{enumerate}
		\item[\rm(a)] $\lambda$ is a L\'evy measure;
		\item[\rm(b)] the measure $\lambda|_r^c$ is finite for each $r>0$, and for some (equivalently, for each) sequence $(\delta_k)_{k\in\N}$ decreasing to $0$ the set $\{\eta(\lambda|_{\delta_k}^c):\, k\in\N\}$ is weakly relatively compact.
	\end{enumerate}
\end{theorem}
In the situation of Theorem \ref{th.Levy-equivalence}, it follows that $\eta(\lambda|_{\delta_k}^c)$ converges weakly to $\eta(\lambda)$; this  follows from Condition (b) and convergence of the corresponding characteristic functions. 

In Hilbert spaces, L\'evy measures can be characterised by an integrability condition: 
\begin{theorem}[\hbox{\cite[Theorem  VI.4.10]{Parthasarathy}}]\label{thm:LevyHilbert} 
	Let $H$ be a separable Hilbert space.
	A $\sigma$-finite measure $\lambda$ satisfying $\la(\{0\})=0$ on $\Borel(H)$ is a L\'evy measure if and only if
	$$ \int_{H} (\n u\n^2\wedge 1) \, \lambda(\d u) < \infty.$$
\end{theorem}

\subsection{Poisson random measures}\label{ss.PRM-Poisson}

L\'evy measures can be characterised by integrability properties of Poisson random measures, which we introduce in the following. For this purpose, let $(\Om,\cF,\PP)$ be a
probability space and let $(E,\cE)$ be a measurable space.
An \emph{integer-valued random measure} is a mapping
$N: \Omega\times\cE\to \N \cup\{\infty\}$
with the following properties:
\begin{enumerate}
	\renewcommand{\labelenumi}{\rm(\roman{enumi})}
	\item For all $B\in\cE$ the mapping $N(B): \omega\mapsto N(\om,B)$ is measurable;
	\item For all $\om\in\Om$ the mapping $N_\omega: B\mapsto N(\om,B)$ is a measure.
\end{enumerate}
The measure $\nu$ on $(E,\mathscr{E})$ defined by
$$\nu(B):= \E (N(B)), \quad B\in \mathscr{E},$$ is called the \emph{intensity measure} of $N$.
An integer-valued random measure $N : \Omega\times\cE\to\N \cup\{\infty\} $
with $\sigma$-finite intensity measure $\nu$ is called a {\em Poisson random measure} (see \cite[Chapter 6]{Cin}) if the
following conditions are satisfied:
\begin{enumerate}
	\renewcommand{\labelenumi}{\rm(\roman{enumi})}
	\addtocounter{enumi}{2}
	\item For all $B \in \cE$ the random variable $N(B)$ is
	Poisson distributed with parameter $\nu(B)$;
	\item For all finite collections of pairwise disjoint sets $B_1,\ldots,B_n$ in $\cE$ the random variables
	$N(B_1)$, \ldots\,, $N(B_n)$ are independent.
\end{enumerate}
In the converse direction,
if $\nu$ is a $\sigma$-finite measure on $\cE$, then by \cite[Prop.\ 19.4]{Sato} there exists a
probability space $(\Om,\cF,\PP)$ and a Poisson random measure $N:\Omega\times\cE\to \N\cup\{\infty\}$ with intensity measure $\nu$.

The {\em Poisson integral} of a measurable function $F: E \to [0,\infty)$ with respect to the Poisson random measure $N$ is the random variable $\int_E F\,\d N$ defined pathwise by
$$ \left(\int_E F(\sigma) \, N(\d \sigma)\right)(\omega) : =  \int_E F(\sigma)\,\d N_\om(\d\sigma),$$
where $N_\om$ is the $\N\cup\{\infty\}$-valued measure of part (ii) of the above definition.

If $N:\Omega\times \cE \to \N\cup\{\infty\}$ is a Poisson random measure with intensity measure $\nu$,
the {\em compensated Poisson random measure} is defined, for $B\in \cE$ with $\nu(B)<\infty$, by
\begin{align*}
	\cp(B):=N(B)-\nu(B).
\end{align*}

In what follows we consider the special case where $E = I\times U$, where $I$ is an interval in $\R_+$ and $U$ is a separable Banach space, and consider Poisson random measures $N:\Omega\times \Borel(I\times U) \to \N\cup\{\infty\}$ whose intensity measure is of the form
$$\nu = \leb\otimes \lambda,$$
where $\leb$ is the Lebesgue measure on the Borel $\sigma$-algebra $\Borel(I)$ and $\lambda$ a $\sigma$-finite measure satisfying $\lambda(\{0\}) = 0$
on the Borel $\sigma$-algebra $\Borel(U)$. {\em These assumptions will always be in force and will not be repeated at every instance.}
The compensated Poisson random measure is then given, for all $t>0$ and all $B\in \Borel(U)$ with $\lambda(B)<\infty$, by
\begin{align*}
	\cp(t,B):=N(t,B)-t\lambda (B),
\end{align*}
using the shorthand notation $$N(t,B):= N((0,t]\times B).$$

For fixed $t>0$, a {\em simple function} with values in another Banach space $V$ is a function $F\colon  (0,t]\times U\to V$ of the form
\begin{align}\label{eq.simple-function}
	F=\sum_{i=1}^m\sum_{j=1}^n  \1_{ (t_i,t_{i+1}]\times B_j}\otimes\, v_{i,j},
\end{align}
where $0=t_1<\dots < t_{m+1}=t$, $v_{i,j}\in V$, and the disjoint sets $B_j\in\Borel(U)$ satisfy $\lambda(B_j)<\infty$ for $i=1,\dots, m$ and $j=1,\dots, n$. Here, and
in what follows, we use the notation $\1_{F}\otimes v$ for the function $t\mapsto \1_F(t)v$.Given $B\in\Borel(U) $, the {\em compensated Poisson integral over $(0,t]\times B$} of a simple function $F\colon (0,t] \times  U\to V$ of the above form is the $V$-valued random variable
\begin{align*}
	I_{B}(F):=\int_{(0,t]\times B} F(s,u)\, \cp(\d s, \d u):=
	\sum_{i=1}^m\sum_{j=1}^n   \cp\big( (t_i\wedge t, t_{i+1}\wedge t],\, B_j\cap B    \big)\otimes v_{i,j}.
\end{align*}
A strongly measurable
function $F\colon (0,t]\times U\to V$ is said to be {\em integrable with respect to $\cp$} if there exists a sequence of simple functions $F_n\colon (0,t]\times  U\to V$ such that
\begin{enumerate}
	\item[(a)] $F_n\to F$  pointwise $(\leb\otimes\lambda)$-almost everywhere;
	\item[(b)] for any $B\in\Borel(U)$,  the sequence $(I_{B}(F_n))_{n\in\N}$ converges in probability as $n\to \infty$.
\end{enumerate}
We say that $F$ is {\em $L^p$-integrable with respect to $\cp$}, where $p\in [1,\infty)$, when the simple functions can be chosen in such a way that $I_{B}(F_n) \in L^p(\Om;V)$ for all $n\in\N$ and the convergence in (b) takes place with respect to the norm of $L^p(\Om;V)$. Here, $L^p(\Om;V)$ denotes the Bochner space of (equivalence classes of) $V$-valued random variables $X$ with $\E(\norm{X}^p)<\infty$; here and in what follows, $V$-valued random variables are always assumed to be strongly $\mathbb{P}$-measurable, i.e., they are $\mathbb{P}$-almost sure limits of a sequence of simple $V$-valued functions (cf. \cite[Chapter 1]{Hytonen-etal}).

It is easily checked that the limit of the sequence $(I_{B}(F_n))_{n\in\N}$ is well defined in the sense that it does not depend on the choice of the approximating sequence $(F_n)_{n\in\N}$.
In this situation, the limit is defined as
\begin{align*}
	I_{B}(F)&:=\int_{(0,t]\times B} F(s,u)\, \cp(\d s,\d u)\\ &:=\lim_{n\to\infty} \int_{(0,t]\times B} F_n(s,u)\, \cp(\d s,\d u) = \lim_{n\to\infty}I_{B}(F_n).
\end{align*}
If $F$ is $L^p$-integrable with respect to $\widetilde N$, then the limit $I_{B}(F)$ belongs to $L^p(\Om;V)$.

For
$V=\R$, the space of integrable functions can be explicitly characterised as follows.

\begin{theorem}[\hbox{\cite[Lemma 12.2 and 12.3]{Kallenberg}, \cite{Rajput-Rosinski}}]\label{th.PRM}
	Let $U$ be a separable Banach space, and consider
	a measurable function $F\colon (0,t]\times U\to \R$ for some $t>0$.
	\begin{enumerate}[\rm(1)]
		\item $F$ is integrable with respect to $N$ if and only if
		$$
		\int_{(0,t]\times U} (\abs{F(s,u)}\wedge 1)\, \d s\,\lambda(\d u)<\infty.$$
		\item
		$F$ is integrable with respect to $\cp$ if and only if
		$$
		\int_{(0,t]\times U} (\abs{F(s,u)}\wedge \abs{F(s,u)}^2)\, \d s\,\lambda(\d u)<\infty.$$
	\end{enumerate}
	In the second case, the characteristic function $\phi_{I_{B}(F)}\colon\R\to \C$ of the real-valued random variable $I_{B}(F)$ is given, for $\beta\in\R$, by
	$$
	\phi_{I_{B}(F)}(\beta)=
	\exp\left( \int_{(0,t]\times B}\left( e^{i\beta F(s,u)}-1-i\beta F(s,u)\right)\, \lambda(\d u)\,\d s\right).
	$$
\end{theorem}

This theorem has the following straightforward vector-valued corollary.

\begin{corollary}\label{co.integrable-cp}
	If $F\colon (0,t]\times U\to V$ is integrable with respect to $\cp$, then for all  $B\in\Borel(U)$ the characteristic function $\phi_{I_{B}(F)}\colon V^\ast\to \C$ of $I_{B}(F)$ is given, for $v^\ast\in V^\ast$, by
	$$
	\phi_{I_{B}(F)}(v^\ast)=
	\exp\left( \int_{(0,t]\times B}\left( e^{i\scapro{F(s,u)}{v^\ast}}-1-i\scapro{F(s,u)}{v^\ast} \right)\, \lambda(\d u)\,\d s\right).
	$$
	
\end{corollary}
\begin{proof}
	We choose a sequence $(F_n)_{n\in\N}$ of simple functions $F_n\colon (0,t]\times U\to V$ converging to $F$ in $V$ $(\leb\otimes\lambda)$-almost everywhere such that the sequence $(I_{B}(F_n))_{n\in\N}$ converges in probability for all $B\in\Borel(U)$. Denoting the limit by $I_{B}(F)$, for fixed  $B\in\Borel(U)$ and $v^\ast\in V^\ast$, it follows that $\scapro{F_n(\cdot,\cdot)}{v^\ast}$ converges to $\scapro{F(\cdot,\cdot)}{v^\ast}$  $(\leb\otimes\lambda)$- almost everywhere in $(0,t]\times U$, and from
	$$
	\int_{(0,t]\times B}\scapro{F_n(s,u)}{v^\ast}\,
	\cp(\d s, \d u) = \Big\langle  \int_{(0,t]\times B} F_n(s,u)\,
	\cp(\d s, \d u) , v^\ast \Big\rangle,
	$$
	it follows that the sequence $(I_{B}(\scapro{F_n(\cdot,\cdot)}{v^\ast}))_{n\in\N}$ converges in probability to the real-valued random variable $\scapro{I_{B}(F)}{v^\ast}$.
	We conclude that $\scapro{F(\cdot,\cdot)}{v^\ast}$ is integrable with respect to $\cp$ and,
	for all  $B\in\Borel(U)$,
	\begin{align}\label{eq.integral-commuting}
		\int_{(0,t]\times B}\scapro{F(s,u)}{v^\ast}\,
		\cp(\d s, \d u)
		=\Big\langle  \int_{(0,t]\times B} F(s,u)\,
		\cp(\d s, \d u) , v^\ast\Big\rangle.
	\end{align}
	Theorem \ref{th.PRM}
	implies that the characteristic function of the real-valued random variable $I(v^\ast):= I_{B}(\scapro{F(\cdot,\cdot)}{v^\ast})$
	is, for $\beta\in\R$, given by
	$$
	\phi_{I(v^\ast)}(\beta)=\exp\left( \int_{(0,t]\times B}\left( e^{i\beta \scapro{F(s,u)}{v^\ast}}-1-i\beta \scapro{F(s,u)}{v^\ast}\right)\, \lambda(\d u)\,\d s\right).
	$$
	Letting $I:=I_{B}(F)$, it follows from \eqref{eq.integral-commuting} that
	$$\phi_I(v^\ast)=\EE[e^{i\scapro{I}{v^\ast}}]=\EE[e^{i I(v^\ast)}]=\phi_{I(v^\ast)}(1).$$
	This completes the proof.
\end{proof}

It is worth pointing out that by choosing $U=V$, $t=1$, $B=B_U$ and $F(s,u)=u$ for all $s\in (0,1]$ and $u\in B_U$, then the distribution of $I_B(F)$ is $\mu(\lambda|_1)$. This will be used in the proof of Theorem \ref{th.Bochner-Levy-measure}.

\begin{lemma}\label{le.integrability-L1}
	Every function $F\colon (0,t]\times U\to V$ belonging to $L^1_{\leb\otimes\lambda}((0,t] \times\, U;V)$ is integrable with respect to $\cp$.
\end{lemma}
\begin{proof}
	Let $F\colon (0,t]\times\, U\to V$ be a simple function of the form \eqref{eq.simple-function}. Recalling that $\la$ is the intensity measure of $N$, it follows for any $B\in\Borel(U)$ that
	\begin{equation}\label{eq.expectation-simple-integral}\begin{aligned}
			&\EE\left[ \left\|\int_{(0,t]\times B} F(s,u)\, \cp(\d s,\d u)\right\|\right]\\
			&\qquad\qquad \le 	\sum_{i=1}^m\sum_{j=1}^n \EE\left[  N\big( (t_i\wedge t, t_{i+1}\wedge t],\, B_j\cap B    \big)\right]
			\norm{v_{i,j}}
			\\ & \qquad\qquad\qquad\qquad + \big( t_{i+1}\wedge t- t_i\wedge t) \lambda(B_j\cap B)
			\norm{v_{i,j}}\\
			& \qquad\qquad = 2 	\sum_{i=1}^m\sum_{j=1}^n \big( t_{i+1}\wedge t- t_i\wedge t) \lambda(B_j\cap B)
			\norm{v_{i,j}}\\
			&\qquad\qquad = 2\int_{(0,t]\times B} \norm{F(s,u)}\, \lambda(\d u)\, \d s.
		\end{aligned}
	\end{equation}
	Now let  $F\colon (0,t]\times U\to V$ be an arbitrary function in $L^1_{\leb\otimes\lambda}((0,t]\times U;V)$. Then there exists a sequence $(F_n)_{n\in\N}$ of simple functions converging to $F$ in $L_{\leb\otimes \lambda}^1((0,t]\times U;V)$; by a routine argument, we may assume that this sequence also converges to $F$ pointwise $(\leb\otimes\lambda)$-almost everywhere in $V$. Since \eqref{eq.expectation-simple-integral} shows that the integrals $I_{t,B}(F_n)$ converge in mean and thus in probability, it follows that $F$ is integrable with respect to $\cp$.
\end{proof}

\subsection{$\gamma$-Radonifying operators}

Consider a Hilbert space $H$ with inner product $(\cdot|\cdot)$ and a Banach space $V$, and denote by $\mathscr{L}(H,V)$ the space of bounded linear operators from $H$ to $V$. The subspace in $\mathscr{L}(H,V)$ consisting of all finite rank operators from $H$ to $V$ is denoted by $H\otimes V$. We will use the notation
$h\otimes v$ for the rank one operator that sends an element $h'\in H$ to the element $(h'|h) v\in V$. By a standard orthogonalisation argument, any finite rank operator $T\in\mathscr{L}(H,V)$ can be expressed as
$$ T = \sum_{n=1}^N h_n\otimes v_n, $$
where $N\ge 1$, the sequence $(h_n)_{n=1}^N$ is orthonormal in $H$, and $(v_n)_{n=1}^N$ is some sequence in $V$.
We introduce $\gamma(H,V)$ as the completion of the space of finite rank operators from $H$ to $V$ under the norm
$$ \left\| \sum_{n=1}^N h_n\otimes v_n\right\|_{\gamma(H,V)}^2 := \E \left\| \sum_{n=1}^N \gamma_n\otimes v_n\right\|^2 ,$$
where $(\gamma_n)_{n=1}^N$ is a sequence of independent, real-valued, standard normally distributed random variables.
This norm is independent of the particular representation of the operator as a sum of finite rank operators, provided the sequence $(h_n)_{n=1}^N$ used in the representation is orthonormal in $H$.
The identity mapping $h\otimes v\mapsto h\otimes v$ is extended to a contractive embedding from $\gamma(H,V)$ into $\mathscr{L}(H,V)$. Consequently, elements of $\gamma(H,V)$ can be identified with bounded linear operators from $H$ to $V$. These operators are called {\em $\gamma$-radonifying operators}.

For comprehensive insights into $\gamma$-radonifying operators, the reader is referred to \cite[Chapter 9]{Hytonen-etal2} and the review paper \cite{NeeCMA}.

Spaces of $\gamma$-radonifying operators enjoy the following {\em ideal property} (see \cite[Theorem 9.1.10]{Hytonen-etal2}):
Given Hilbert spaces $H_1, H_2$ and Banach spaces $V_1, V_2$, for every $R\in \mathscr{L}(H_1, H_2)$, $S\in \gamma(H_2,V_2)$, and $T\in \mathscr{L}(V_2,V_1)$, it holds that
$TSR\in \gamma(H_1,V_1)$ and
\begin{align}\label{eq:ideal-property}\|TSR\|_{\gamma(H_1,V_1)}\leq \|T\|_{\mathscr{L}(V_2,V_1)} \, \|S\|_{\gamma(H_2,V_2)} \, \|R\|_{\mathscr{L}(H_1,H_2)}.
\end{align}

We will need various other standard results on $\gamma$-radonifying operators; these will be quoted as soon as the need arises. Let us finally give some examples:

\begin{example}[Hilbert spaces]
	When both $H$ and $V$ are Hilbert spaces we have a natural isometric isomorphism $$\gamma(H,V) = \mathscr{L}_2(H,V),$$
	the space of Hilbert-Schmidt operators from $H$ to $V$.
\end{example}

\begin{example}[$L^p$-spaces]
	When $H$ is a Hilbert space and $V = L^p(S,\mu)$, where $(S,\mu)$ is a $\sigma$-finite measure space and $p\in [1,\infty)$,
	the mapping $J:L^p(S,\mu;H)\to \gamma(H,L^p(S,\mu))$ given by  $(Jf)h : =\scapro{f(\cdot)}{h}$
	defines an isomorphism of Banach spaces
	\begin{equation}\label{eq:SqFc}\gamma(H,L^p(S,\mu))\simeq_p L^p(S,\mu;H)
	\end{equation}
	with isomorphism constants only depending on $p$.
	In the particular case when $H$ is an $L^2$-space, the spaces on the right-hand side are
	usually referred to as {\em spaces of square functions} and play a prominent role in Harmonic Analysis.
	It is worth mentioning that the isomorphism \eqref{eq:SqFc} extends to Banach lattices $V$ with finite cotype.
\end{example}

\section{L\'evy measures on $L^p$-spaces}

We will now specialise to $V = L^p_{\mu}:=L^p_{\mu}(S):= L^p(S,\mu)$ for some $p\in (1,\infty)$ and a measure space $(S,\S,\mu)$.
It will always be assumed that $\mu$ is $\sigma$-finite and that $L^p_{\mu}$ is separable; these assumptions are for example satisfied if  the measure space $(S,\S,\mu)$ is {\em $\mu$-countably generated} according to \cite[Proposition 1.2.29]{Hytonen-etal}.
The $\sigma$-finiteness of $\mu$ implies that the various uses of Fubini's theorem in this paper are justified, and also that we may identify the dual space $(L_\mu^p)^*$ with $L_\mu^{p'}$, where $\frac1p+\frac1{p'}=1$ (although, as is well known, $\sigma$-finiteness is not needed for this identification in the regime $p\in (1,\infty)$). The separability of $L_\mu^p$ implies that the mapping $f\mapsto f$, viewed as a function from $(L_\mu^p,\Borel(L_\mu^p))$ to itself, is strongly measurable (by Pettis's measurability theorem, see \cite[Theorem 1.1.6]{Hytonen-etal}).

We recall that when $U$ is a Banach space, $L^p(S;U)$ denotes the Banach space of all (equivalence classes of) strongly $\mu$-measurable $f\colon S\to U$ for which $$\n f\n_{L^p(S;U)} := \left(\int_S \n f(s)\n^p \,\mu(\d s)\right)^{1/p}$$
is finite. We sometimes wish to emphasise the measure $\mu$, in which case we write  $L_\mu^p(S;U)$
instead of $L^p(S;U)$. With this notation, $L_\mu^p(S;\R) = L_\mu^p$.

As before, let $U$ be a separable Banach space, and let $\lambda$ a $\sigma$-finite measure on $\Borel(U)$
which satisfies $\la(\{0\})=0$. Let $N$ denote a Poisson random measure on $\Rp\times U$ with intensity measure $\leb\otimes\lambda$, and let $\cp$ denote the associated compensated Poisson random measure. For fixed $t>0$ and for a simple function $F\colon (0,t]\times U\to L_{\mu}^p(S)$ of the form \eqref{eq.simple-function}, for each  $B\in\Borel(U)$, the compensated Poisson integral
\begin{align*}
	I_{B}(F):=\int_{(0,t]\times B} F(r,u)\, \cp(\d r, \d u)
\end{align*}
is defined as in Subsection \ref{ss.PRM-Poisson}.  As a special case of the result in \cite{Dirksen}, for simple functions $F:(0,t]\times  U\to L^p_\mu(S)$, $B\in\Borel(U)$, exponents $p\in (1,\infty)$,  we have the equivalence of norms
\begin{align}\label{eq.Dirksen}
	\left(\EE\left[ \sup_{0<s\le t}\norm{ \int_{(0,s]\times B} F(r,u)\, \cp(\d r,\d u) }_{L^p_\mu(S)}^p\right]\right)^{1/p}
	\simeq_{p} \norm{F \1_{(0,t]\times B}}_{\I_{p}},
\end{align}
where
\begin{align*}
	\I_p:=
	\begin{cases}
		\S_\lambda^p + \D_\lambda^p & \hbox{if} \ 1<p<2,\\
		\S_\lambda^p\cap \D_\lambda^p &  \hbox{if} \  2\le p<\infty,
	\end{cases}
\end{align*}
with
\begin{align*}
	\S_\lambda^p = L^p_\mu(S,L^2_{\leb\otimes\lambda}(\Rp\times U)), \qquad
	\D_\lambda^p = L^p_{\leb\otimes\lambda}(\Rp\times U;L^p_\mu(S)),
\end{align*}
and
\begin{align*}
	\norm{F}_{\S_\lambda^p + \D_\lambda^p} & :=
	\inf\left\{ \norm{F_1}_{\S_\lambda^p}+\norm{F_2}_{\D_\lambda^p}:
	F=F_1+F_2, \,F_1\in \S_\lambda^p, \,F_2\in \D_\lambda^p\right\},  \\
	\norm{F}_{\S_\lambda^p \cap \D_\lambda^p} & := 	\max\left\{ \norm{F}_{\S_\lambda^p}, \norm{F}_{\D_\lambda^p}\right\}.
\end{align*}
Here, both $\S_\lambda^p$ and $\D_\lambda^p$ are viewed as Banach spaces of (equivalence classes of)
measurable real-valued functions on $\Rp\!\times S\times U$.
Explicitly, the norms in these spaces are defined by
\begin{align*}
	\norm{F}_{\S_\lambda^p}^p &:  = \norm{F}_{L^p_\mu(S,L^2_{\leb\otimes\lambda}(\Rp\times U))}^p =\int_S \left( \int_{(0,\infty)\times U}\abs{F(t,u)(s)}^2\, \lambda(\d u)\,\d t\right)^{p/2}\mu(\d s) ,\\
	\norm{F}_{\D_{\lambda}^p}^p&: = \norm{F}_{L^p_{\leb\otimes\lambda}(\Rp\times U;L^p_\mu(S))}^p =\int_{(0,\infty)\times U} \int_S |F(t,u)(s)|^p\, \mu(\d s) \,  \lambda(\d u)\,\d t,
\end{align*}
where the second identity in the first line follows from \cite[Theorem 17, p. 198]{Dunford-Schwartz-1} or  \cite[Proposition 1.2.25]{Hytonen-etal}, which allow us to take the point evaluation with respect to $s$ inside the integral.

\begin{remark}\label{re.Banach-function}
	For later use  we observe that $ \I_p$ is a Banach function space in the sense of Bennett and Sharpley \cite{BenSha}; this follows from \cite[Problems 4 and 5, page 175]{BenSha}. In particular, if $0\le f\le g$ almost everywhere with $g\in \I_p$, then $f\in \I_p$ and $\norm{f}_{\I_p}\le \norm{g}_{\I_p}$.
	Keeping in mind that we are assuming $p\in (1,\infty)$, the spaces $\I_p$ are reflexive as Banach spaces; this follows from, e.g., \cite[Corollary IV.1.2]{DieUhl} along with the easy fact that if $(X_0,X_1)$ is an interpolation couple of reflexive Banach spaces, then the spaces $X_0\cap X_1$ and $X_0+X_1$ are reflexive. By \cite[Corollary 1.4.4]{BenSha}, this implies that the norm of $\I_p$ is absolutely continuous.
\end{remark}

\begin{lemma}\label{lem:simple-Iq-dense}
	The class of simple functions
	$F\colon \Rp\times U\to L_\mu^p(S)$ is dense in ${\mathcal I}_p$.
\end{lemma}
\begin{proof}
	This follows from \cite[Theorem 1.3.11]{BenSha}.
\end{proof}

As a consequence of the lemma above, for all $t\in\R_+$ and $B\in \Borel(U)$, the compensated Poisson integral $I_{B}(F)$ can now be defined
by a standard density argument for all  strongly measurable functions $F\colon (0,t]\!\times U\to L_\mu^p(S)$ such that $F\1_{(0,t]\times B}\in  {\mathcal I}_p$.

We now set $U=L_\mu^p(S)$ to obtain the following characterisation of L\'evy measures in $L_\mu^p(S)$.
\begin{theorem}\label{th.Bochner-Levy-measure}
	A $\sigma$-finite measure $\lambda$ on $\Borel(L^p_\mu)$ with $\lambda(\{0\})=0$  is a L\'evy measure if and only if $\lambda|_r^c$ is a finite measure for all $r>0$ and, moreover,
	\begin{enumerate}[\rm (1)]
		\item if $p\in [2,\infty)$, it satisfies
		\begin{align*}
			\max\left\{  \int_{S} \left(\int_{B_{{L_\mu^p}}}  \abs{f(s)}^2\, \lambda(\d f) \right)^{p/2} \mu(\d s)     ,\, \int_{B_{{L_\mu^p}}} \norm{f}^p_{L^p_\mu}\, \lambda(\d f) \right\} <\infty.
		\end{align*}
		\item if $p\in (1,2)$, it satisfies
		\begin{align*}
			\inf\left\{ \int_{\S} \left(\int_{B_{{L_\mu^p}}}  \abs{F_1(f)(s)}^2\, \lambda(\d f) \right)^{p/2} \mu(\d s)
			+  \int_{B_{{L_\mu^p}}} \norm{F_2(f)}^p_{L_\mu^p}\,\lambda(\d f)   \right\}<\infty,
		\end{align*}
		where the infimum is taken over all functions $F_1\in {\mathcal S}^p_\lambda$ and $F_2\in \D_{\lambda}^p$ with $F_1(f)+F_2(f)=f$ for all $f\in B_{{L_\mu^p}}=\{g\in L_\mu^p:\norm{g}_{\L_\mu^p}\le 1\}$.
	\end{enumerate}
\end{theorem}

\begin{proof}
	For the closed unit ball $B_{L_\mu^p}=\{f\in L^p_\mu:  \norm{f}_{L_\mu^p}\le  1\}$, we define the function
	$$G\colon (0,1]\times L_\mu^p\to L_\mu^p, \qquad G(t,f) =f \1_{B_{L_\mu^p}}(f), $$
	and, letting $D_\delta:=\{f\in L^p_\mu: \delta < \norm{f}_{L_\mu^p}\le  1\}$ for some $\delta\in (0,1)$, we introduce analogously
	$$G_\delta\colon (0,1]\times L_\mu^p\to L_\mu^p, \qquad G_\delta(t,f) =f \1_{D_\delta}(f).$$
	Note, that $G$ and $G_\delta$ do not depend on $t$, but we left the previous notation for consistency. 
	\smallskip
	
	{\rm `If'}: Let $N$ be a Poisson random measure with intensity $\leb\otimes\lambda$. Assume first that the support of $\lambda$ is contained in the closed unit ball $B_{{L_\mu^p}}$.
	The assumed integrability conditions guarantee that $G$ belongs to ${\mathcal I}_p$ and thus  we can define the random variable
	\begin{align*}
		X:=\int_{(0,1]\times B_{{L_\mu^p}}} f \, \cp(\d s,\d f).
	\end{align*}
	Corollary \ref{co.integrable-cp} shows that the probability distribution of $X$ coincide with $\eta(\lambda)$, which shows that $\lambda$ is a L\'evy measure by its very definition.

	For the general case of a $\sigma$-finite measure $\lambda$ with $\la(\{0\})=0$, we apply the decomposition $\lambda=\lambda|_{1}+\lambda|_{1}^c$
	of \eqref{eq:decompose-lambda}. The measure $\lambda|_{1}$ is a L\'evy measure by the first part, and $\lambda|_{1}^c$ is a L\'evy measure since it is finite by assumption. Now \cite[Proposition 5.4.9]{Linde} guarantees that $\lambda$ is a L\'evy measure.
	
	\smallskip
	{\rm `Only if':} Assume that $\lambda$ is a L\'evy measure. Theorem  \ref{th.Levy-equivalence}
	implies that $\lambda|_r^c$ is a finite measure for all $r>0$.
	
	To establish the integrability conditions we can assume that the L\'evy measure $\lambda$ has support in the closed unit ball $B_{{L_\mu^p}}$. Let $N$ be a Poisson random measure with intensity $\leb\otimes\lambda$. Let   $(\delta_k)_{k\in\N}\subseteq (0,1)$ be an arbitrary sequence decreasing to $0$.  Since $\lambda(D_{\delta_k})<\infty$,
	Lemma \ref{le.integrability-L1} guarantees that $G_{\delta_k}$ is integrable with respect to $\cp$ for each $k\in\N$. Thus, we can define the random variables
	\begin{align*}
		X_k:=\int_{(0,1]\times  D_{\delta_k}} f \, \cp(\d s,\d f)\quad\text{for } k\in\N.
	\end{align*}
	Since $X_k$ has the same distribution as $\eta(\lambda|_{\delta_k}^c)$ by Corollary \ref{co.integrable-cp}, 
	Theorem  \ref{th.Levy-equivalence} implies that  $(X_k)_{k\in\N}$ converges weakly to $\eta(\lambda)$ in the space of Borel probability measures on $\Borel(L^p_\mu)$. Letting
	$Y_k:=X_k-X_{k-1}$ for $k\in\N$, with $X_{0}:=0$, it follows  that the random variables $Y_k$ are independent as the sets $ D_{\delta_k}\,\setminus\,  D_{\delta_{k-1}}$ are disjoint for all $k\in \N$. Since $X_k=Y_1+\dots +Y_k$ is a sum of independent random variables converging weakly,  L\'evy's theorem in Banach spaces (see, e.g., \cite[Theorem  V.2.3, page 268]{Vaketal}) implies that the random variables $X_k$ converge almost surely to a random variable $X$, which must have distribution  $\eta(\lambda)$. Since $\E(\norm{X}^p)<\infty$ by \cite[Corollary 3.3]{deAcosta-1980}, it follows from \cite[Corollary 3.3]{Hoffmann-Jorgensen-sums} that $X_k\to X$ in $L^p(\Omega;L_\mu^p)$ as $k\to\infty$. The isometry \eqref{eq.Dirksen} implies that $G_{\delta_k}\to G$ in $\I_p$ as $k\to\infty$.
\end{proof}

\begin{example}
	If $p=2$, Fubini's theorem implies
	\begin{align*}
		\int_{S} \int_{B_{{L_\mu^2}}} \abs{f(s)}^2\, \lambda(\d f)\,\mu(\d s)
		& =  \int_{B_{{L_\mu^2}}} \int_S \abs{f(s)}^2\, \mu(\d s) \,  \lambda(\d f)
		=  \int_{B_{{L_\mu^2}}}  \norm{f}^2_{L^2_\mu} \,  \lambda(\d f).
	\end{align*}
	Consequently, the two integrals in part (a) of Theorem \ref{th.Bochner-Levy-measure} coincide. Taking into account the condition that $\lambda|_r^c$ is finite for all $r>0$, it follows that a $\sigma$-finite measure $\lambda$ on $L^2_\mu$ with $\lambda(\{0\})=0$ is a L\'evy measure if and only if
	$$\int_{L^2_\mu} \left(\norm{f}^2_{L_\mu^2}\wedge 1\right) \,\lambda(\d f)<\infty. $$
	This corresponds to the well known characterisation of L\'evy measures on Hilbert spaces in Theorem \ref{thm:LevyHilbert}.
\end{example}

\begin{example} \label{ex.sequence-space-larger-2}
	In this Example, we consider the sequence space $\ell^p = \ell^p(\N)$ with $p\in [2,\infty)$. The canonical sequence of unit vectors in $\ell^p$ is denoted by $(e_k)_{k\in\N}$.
	Let  $\lambda$ be a $\sigma$-finite measure on
	$\Borel(\ell^p)$ with $\lambda(\{0\})=0$ and $\lambda|_r^c$ finite for all $r>0$.
	Then the condition in part (a) of Theorem \ref{th.Bochner-Levy-measure} is satisfied if
	\begin{align*}
		\sum_{k=1}^\infty \left(\int_{B_{{\ell^p}}}\scapro{f}{e_k}^2\, \lambda(\d f)\right)^{p/2}<\infty
		\qquad\text{and}\qquad
		\int_{B_{{\ell^p}}} \norm{f}^p_{\ell^p} \, \lambda(\d f)<\infty.
	\end{align*}
	Again taking into account the assumption that $\lambda|_r^c$ is finite for all $r>0$, we can conclude that  $\lambda$ is a L\'evy measure if and only if
	\begin{align*}
		\sum_{k=1}^\infty \left(\int_{B_{{\ell^p}}}\scapro{f}{e_k}^2\, \lambda(\d f)\right)^{p/2}<\infty
		\qquad\text{and}\qquad
		\int_{\ell^p} \left(\norm{f}^p_{\ell^p}\wedge 1\right) \, \lambda(\d f)<\infty.
	\end{align*}
	This characterisation  coincides with the result derived in \cite[Theorem 3]{Yurinskii}.
\end{example}

\begin{example}  \label{ex.sequence-space-smaller-2}
	In this example, we consider the sequence space $\ell^p = \ell^p(\N)$ for $p\in (1,2)$.
	Let  $\lambda$ be $\sigma$-finite measure on
	$\Borel(\ell^p)$ with $\lambda(\{0\})=0$ and $\lambda|_r^c$ finite for all $r>0$.
	Theorem \ref{th.Bochner-Levy-measure} shows that $\lambda$ is a L\'evy measure if and only if
	\begin{align*}
		\inf\left\{\sum_{k=1}^\infty  \scapro{\int_{B_{{\ell^p}} }\abs{F_1(f)}^2\, \lambda(\d f)}{e_k}^{p/2}
		+ \int_{B_{{\ell^p}}} \norm{F_2(f)}_{\ell^p}^p \, \lambda(\d f)\right\}<\infty,
	\end{align*}
	where the infimum is taken over all functions $F_1\in \S_\lambda^p$ and $F_2\in \D_\lambda^p$ with $F_1(f)+F_2(f)=f$ for all $f\in B_{{\ell^p}}=\{g\in \ell^p:\norm{g}_{\ell^p}\le 1\}$.
	
	If  we take $F_1=\Id_{\ell^p} \1_{B_{{\ell^p}}}$ and $F_2=0$ or $F_1=0$ and $F_2=\Id_{\ell^p} \1_{B_{{\ell^p}}}$   then we obtain the sufficient conditions
	\begin{align*}
		\sum_{k=1}^\infty \left(\int_{B_{{\ell^p}}}\scapro{f}{e_k}^2\, \lambda(\d f)\right)^{p/2}<\infty
		\qquad\text{or}\qquad
		\int_{B_{{\ell^p}}}\norm{f}^p_{\ell^p}\, \lambda(\d f)<\infty.
	\end{align*}
	The second condition is known as a sufficient condition due to the fact that the space $\ell^p$ is of type $p$ for $p\in [1,2]$; see \cite{AraujoGine2}.
\end{example}

With the same methods as in Theorem \ref{th.Bochner-Levy-measure}, but using the $L^p$-estimates in martingale type and cotype spaces from \cite{Dirksen-Maas-Neerven}, one can show that if $U$ is a separable Banach space with martingale type $p\in (1,2]$ and $\lambda$ is a $\sigma$-finite measure on $\Borel(U)$ with $\la(\{0\}) =0$, then
$$ \int_U (\n u\n^p \wedge 1)\,\lambda(\d u) < \infty$$
implies that $\lambda$ is a L\'evy measure. In the converse direction, if $U$ has martingale cotype $q\in [2,\infty)$, and if $\lambda$ is a L\'evy measure on $\Borel(U)$, one can similarly show that
$$ \int_U (\n u\n^q \wedge 1)\,\lambda(\d u) < \infty.$$
We leave the details to the reader, since these results are already covered, with a different method of proof, in \cite{AraujoGine2}.

\section{L\'evy measures on UMD Banach spaces}

The aim of this section is to extend the results of the preceding section to UMD-spaces.
This class of Banach spaces plays a prominent role in stochastic analysis, where it provides the correct setting for Banach space-valued martingale theory (see \cite{Hytonen-etal} and the references therein) and the theory of stochastic integration (see \cite{Hytonen-etal, Neerven-etal=integration, NVW12, Hytonen-etal4} and the references therein), and in harmonic analysis (see \cite{Hytonen-etal3} and the references therein), in that several of the main theorems in these areas admit extensions to the functions with values in a Banach space $X$ if and only of $X$ is a UMD-space. For example, the Hilbert transform on $L^p(\R)$ extends boundedly to $L^p(\R;X)$ if and only if $X$ is a Banach space, and a similar characterisation holds for the It\^o isometry.

A Banach space $V$ is said to be a {\em UMD-space} when, for some (or equivalently, any) given $p\in (1,\infty)$, there exists a constant $\beta_{p,V}\ge 1$ such that for
every $V$-valued martingale difference sequence $(d_j)_{j=1}^n$ and every $\{-1, 1\}$-valued sequence $(\epsilon_j)_{j=1}^n$ we have
$$
\left(\E\left\|\sum_{j=1}^n \epsilon_j d_j \right\|^p\right)^{1/p}
\leq \beta_{p,V} \,
\left(\E\left\|\sum_{j=1}^n d_j\right\|^p\right)^{1/p}.
$$
Examples of UMD-spaces include Hilbert spaces and the spaces $L^p(S,\mu)$ with $1<p<\infty$, for arbitrary measure spaces $(S,\mu)$. Additionally, when $V$ is a UMD-space, then for any $1<p<\infty$, the Bochner spaces $L^p(S,\mu;V)$ are UMD-spaces. A comprehensive treatment of UMD-spaces is offered in \cite{Hytonen-etal} and the references therein.

Let $V$ be a UMD-space and $M\colon \Rp\times\Omega\to V$ a purely discontinuous martingale.
For a fixed time $t>0$, define path-wise a random operator $J_{\Delta M}: \ell^2((0,t]) \to V$ by
\begin{equation*}
	J_{\Delta M}h :=\sum_{s\in (0,t]} h_s \Delta M(s), \quad h = (h_s)_{s\in (0,t]}\in \ell^2((0,t]),
\end{equation*}
where $\Delta M$ is the jump process associated with $M$,
and $\ell^2((0,t])$ is the Hilbert space of all mappings $f:(0,t]\to \R$ satisfying
$$ \n f\n_{\ell^2((0,t])}^2 := \sum_{s\in (0,t]} |f(s)|^2 <\infty,$$
the sum on the right-hand side being understood as the supremum of all sums $\sum_{s\in F} |f(s)|^2 $ with $F\subseteq (0,t]$ finite.

Now let $p\in [1,\infty)$ be given and $M$ be a $V$-valued purely discontinuous martingale. It is shown in \cite[Theorem  6.5]{Yaroslavtsev-20}
that
$M$ is an $L^p$-martingale if and only if
for each $t\ge 0$ we have
$J_{\Delta M}\in \gamma(\ell^2((0,t]),V)$ almost surely and
$$ \E \left[ \norm{J_{\Delta M}}_{\gamma(\ell^2((0,t]),V)}^p\right] <\infty,$$
where $\gamma(\ell^2((0,t]),V)$ is the Banach space of $\gamma$-radonifying operators from  $\ell^2((0,t])$ to $V$, and that, moreover, in this situation one has the equivalence of norms
\begin{equation}\label{eq.UMD-iso}
	\EE\left[\sup_{0< s\le t} \norm{M(s)}^p\right]
	\simeq_{p,V} \EE\left[ \norm{J_{\Delta M}}_{\gamma(\ell^2((0,t]),V)}^p\right].
\end{equation}

In the remainder of this section, we let $U$ be a separable Banach space, and consider
a Poisson random measure $N$ with intensity measure $\leb\otimes\lambda$
for a $\sigma$-finite measure $\lambda$ on $\Borel(U)$ with $\lambda(\{0\})=0$. The compensated Poisson random measure is denoted by $\cp$. 
Our aim is to apply \eqref{eq.UMD-iso} to obtain an $L^p$-bound (see \cite[Section 7.2]{Yaroslavtsev-20}) for martingales $M$ of the form
\begin{equation}\label{def.M}
	M_B(s):=\int_{(0,s]\times B} F(r,u)\, \cp(\d r,\d u), \quad s\in (0,t],
\end{equation}
for simple functions $F\colon (0,t] \times U\to V$ and some $t>0$ and $B\in\Borel(U)$ fixed. This $L^p$-bound will allow us to extend the class of functions integrable with respect to $\widetilde N$ to a more general class of integrands.

For a measurable function $g: (0,t] \times U\to \R$ we write
$ g\in L_N^2((0,t]\times U)$ if for all $\omega\in\Omega$ we have
$$ \n g\n_{L_N^2((0,t]\times U)}(\omega):= \int_{(0,t]\times U} |g(r,u)|^2\, N(\om,\d r, \d u) < \infty.$$
In this way we may interpret the expression $\n g\n_{L_N^2((0,t] \times U)}$ as a nonnegative random variable on $\Omega$.

Now, for a strongly measurable function $F\colon (0,t] \times  U\to V$ introduce the restriction $F_B\colon (0,t] \times  U\to V$ defined by $F_B(r,u):=\1_B(u)F(r,u)$ for the set $B$ defining the martingale $M_B$ in \eqref{def.M}. 
If $F$ satisfies
\begin{align}\label{eq:weakL2}\int_{(0,t]\times U} |\scapro{F_B(r,u)}{v^\ast}|^2\, N(\om,\d r, \d u) <\infty \quad\forall \omega\in \Omega, \ v^\ast\in V^\ast,
\end{align}
we may now define, for every $\omega\in\Omega$, a bounded operator $T_{F_B}(\omega): L_{N(\omega)}^2((0,t]\times U)\to V$ by the Pettis integral (which is well defined by \cite[Theorem 1.2.37]{Hytonen-etal})
\begin{align}\label{eq:Tf}
	T_{F_B}(\omega) g := {\rm (P)\,-}\int_{(0,t]\times U} g(r,u) F_B(r,u)\, N(\om,\d r, \d u).
\end{align}
The Pettis measurability theorem implies that the $V$-valued random variable $\omega\mapsto T_{F_B}(\omega) g$ is strongly measurable. 

Pointwise on $\Omega$, the following identities holds for all $v^\ast\in V^\ast$:
\begin{align*}
	\n J_{\Delta M_B}^\ast v^\ast \n_{\ell^2((0,t])}^2
	& =\sum_{s\in (0,t]} \abs{ \scapro{\Delta M_B(s)}{v^\ast}}^2
	\\ & = \int_{(0,t]\times U} |\scapro{F_B(r,u)}{v^\ast}|^2\, N(\d r, \d u)
	\\ & = \n T_{F_B}^\ast v^\ast\n_{L_N^2((0,t]\times U)}^2,
\end{align*}
the middle identity being a consequence of \cite[Corollary 4.4.9]{Applebaum}.

Hence, as consequence of the comparison theorem for $\gamma$-radonifying operators (see \cite[Theorem 9.4.1]{Hytonen-etal2}), applied pointwise on $\Omega$, we obtain that  $T_{F_B}\in \gamma(L^2_N((0,t]\times U), V)$ almost surely if and only if $J_{\Delta M_B}\in \gamma(\ell^2((0,t]),V)$ almost surely, in which case we have almost surely the  identity of norms
\begin{align}\label{Yaro}  \norm{J_{\Delta M_B}}_{\gamma(\ell^2((0,t]),V)} =
	\norm{T_{F_B}}_{\gamma(L^2_N((0,t]\times U), V)}.
\end{align}
These considerations are key to proving the following theorem.

\begin{theorem}\label{thm.UMD-iso-int}
	Let $V$ be a UMD-space and let $p\in [1,\infty)$. For fixed $t>0$, let $F\colon (0,t]\times  U\to V$ be a strongly measurable function satisfying the weak $L^2$-integrability condition \eqref{eq:weakL2} for all $B\in\Borel(U)$.  Then the following assertions are equivalent:
	\begin{enumerate}[\rm(1)]
		\item\label{it.UMD-iso-int1} $F$ is $L^p$-integrable with respect to $\widetilde N$ and satisfies, for all  $B\in \Borel(U)$,
		\begin{align*}
			\EE\left[ \sup_{0< s\le t} \norm{\int_{(0,s]\times B} F(r,u)\, \cp(\d r,\d u)}^p\right] < \infty;
		\end{align*}
		\item\label{it.UMD-iso-int2} $T_{F_B}$ is in $\gamma(L^2_N((0,t]\times U), V)$ almost surely for all $B\in \Borel(U)$  and 
		\[ \E\norm{T_{F_B}}_{\gamma(L^2_N((0,t]\times U), V)}^p <\infty.\]
	\end{enumerate}
	In this situation, for all  $B\in \Borel(U)$, one has
	\begin{align*}
		\EE\left[ \sup_{0< s\le t} \norm{\int_{(0,s]\times B} F(r,u)\, \cp(\d r,\d u)}^p\right]
		\simeq_{p,V} \E\norm{T_{F_B}}_{\gamma(L^2_N((0,t]\times U), V)}^p,
	\end{align*}
	with constants depending only on $p$ and $V$.
\end{theorem}

\begin{proof}
	\ref{it.UMD-iso-int2}$\implies$\ref{it.UMD-iso-int1}:
	If $F$ is {\em simple} and satisfies the conditions of \ref{it.UMD-iso-int2}, this implication
	follows by combining \eqref{eq.UMD-iso} and \eqref{Yaro}, the point here being that the $L^p$-integrability of $F$ with respect to $\widetilde N$ holds by definition.
	
	Suppose now that $F$ is strongly measurable and satisfies the conditions of \ref{it.UMD-iso-int2}.
	The idea of the proof is to approximate $F$ with simple functions satisfying the conditions of the theorem. To this end, fix $B\in \Borel(U)$ and let $(\calF_n)_{n\in\N}$ and $(\calG_n)_{n\in\N}$ be filtrations generating the Borel $\sigma$-algebras of $\R_+$ and $U$, such that each $\calF_n$ and $\calG_n$ consists of finitely many Borel sets. For each $\omega\in \Omega$, let 
	\[\E_n \colon L_{N(\omega)}^2((0,t]\times U;V)\to L_{N(\omega)}^2((0,t]\times U;V)\]
	denote the (vector-valued) conditional expectation with respect to the product $\sigma$-algebra $\calF_n\times \calG_n$ (see \cite[Chapter 2]{Hytonen-etal}). The functions $$F_{n,B} := \E_n F_B$$ are simple, and  each of them satisfies the conditions of the theorem.
	
	For each $\omega\in\Omega$ and $v^\ast\in V^\ast$, the $L^2$-contractivity of conditional expectations gives
	\begin{align*}\int_{(0,t]\times U} |\scapro{(\E_n F_B)(r,u)}{v^\ast}|^2\, N(\om,\d r, \d u)
		& = \int_{(0,t]\times U} |\E_n\scapro{F_B}{v^\ast}(r,u)|^2\, N(\om,\d r, \d u)
		\\ & = \n \E_n\scapro{F_B}{v^\ast}\n_{L_{N(\omega)}^2((0,t]\times U)}^2
		\\ & \le \n \scapro{F_B}{v^\ast}\n_{L_{N(\omega)}^2((0,t]\times U))}^2 <\infty.
	\end{align*}
	The self-adjointness of $\E_n$, see \cite[Proposition 2.6.32]{Hytonen-etal},
	implies  for each $\omega\in\Omega$ and $g\in L_{N(\omega)}^2((0,t]\times U)$ that
	\begin{align*} T_{F_{n,B}}(\omega)g
		& =  \int_{(0,t]\times B} g(r,u) (\E_n F_B)(r,u)\, N(\om,\d r, \d u)
		\\ & = \int_{(0,t]\times B} (\E_n g)(r,u) F_B(r,u)\, N(\om,\d r, \d u)
		\\ & = (T_{F_B}(\omega) \circ \E_n)g.
	\end{align*}
	Therefore, we conclude  $ T_{F_{n,B}}(\omega) = T_{F_B}(\omega) \circ \E_n \in \gamma(L_{N(\omega)}^2((0,t]\times U),V)$
	by the ideal property \eqref{eq:ideal-property} and, using again that $\E_n$ is contractive,
	\begin{align*}
		\n T_{F_{n,B}}(\omega) \n_{ \gamma(L_{N(\omega)}^2((0,t]\times U),V)}
		& = \n T_{F_B}(\omega) \circ \E_n \n_{ \gamma(L_{N(\omega)}^2((0,t]\times U),V)}
		\\ & \le  \n T_{F_B}(\omega) \n_{ \gamma(L_{N(\omega)}^2((0,t]\times U),V)}.
	\end{align*}
	Next, since $\E_n\to I$ strongly, it follows from \cite[Theorem 9.1.14]{Hytonen-etal2} that
	$$ \lim_{n\to\infty} \n T_{F_B-F_{n,B}}(\omega)\n_{\gamma(L_{N(\omega)}^2((0,t]\times U),V)} = \n T_{F_B}(\omega) - T_{F_{n,B}}(\omega)\n_{\gamma(L_{N(\omega)}^2((0,t]\times U),V)} = 0.$$
	Finally, by monotone convergence,
	$$\lim_{n\to\infty} \E \n T_{F_{n,B}}\n_{\gamma(L_{N}^2((0,t]\times U),V)}^p
	= \n T_{F_B}\n_{\gamma(L_{N}^2((0,t]\times U),V)}^p.
	$$
	Since the theorem holds for each of the $F_{n,B}$,  using routine arguments the theorem now follows by letting $n\to\infty$.
	
	\smallskip
	
	\ref{it.UMD-iso-int1}$\implies$\ref{it.UMD-iso-int2}:
	Suppose that $F$ is strongly measurable and satisfies the conditions of \ref{it.UMD-iso-int1}.
	Choose a sequence of simple functions $F_n\colon (0,t]\times U\to V$ such that
	$F_n\to F$  pointwise $(\leb\otimes\lambda)$-almost everywhere and, for any $B\in\Borel(U)$, one has
	$$
	\int_{(0,t]\times B} F_n(r,u)\, \cp(\d r,\d u) \to \int_{(0,t]\times B} F(r,u)\, \cp(\d r,\d u)
	$$
	in $L^p(\Om;V)$ as $n\to \infty$. By Doob's inequality, one then also has
	$$ \lim_{n,m\to\infty}\EE\left[ \sup_{0< s\le t} \norm{\int_{(0,s]\times B} (F_n(r,u) - F_m(r,u)) \, \cp(\d r,\d u)}^p\right] = 0.$$
	Letting $F_{n,B}:=\1_BF_n$ and $F_B:=\1_BF$, it follows from \eqref{eq.UMD-iso} and \eqref{Yaro} that
	$$\lim_{n,m\to\infty} \E\norm{T_{F_{n,B}} - T_{F_{m,B}}}_{\gamma(L^2_N((0,t]\times U), V)}^p = 0.$$
	Passing to a subsequence, 
	we may assume that, for almost all $\om\in\Om$,
	$$\lim_{n\to\infty} \n T_{F_{n,B}}(\om) -T_{F_{m,B}}(\om) \n_{\gamma(L^2_{N(\om)}((0,t]\times U), V)} = 0.$$
	By completeness of $\gamma(L^2_{N(\om)}((0,t]\times U), V)$ it follows that, for almost all $\om\in\Om$,
	the limit
	$$ T_B(\om):= \lim_{n\to\infty} T_{F_{n,B}}(\om)$$ exists in $\gamma(L^2_{N(\om)}((0,t]\times U), V)$.
	Then also, for any $v^\ast\in V^\ast$,
	$$ (T_{F_{n,B}}(\om))^\ast v^\ast \to (T_B(\om))^\ast v^\ast \ \ \hbox{in $L^2_{N(\om)}((0,t]\times U)$}.$$
	Hence, for all $g\in L^2_{N(\om)}((0,t]\times U)$ and $v^\ast\in V^\ast$, it follows that
	$$ \scapro{T_B(\om)g}{v^\ast} =\lim_{n\to\infty} \scapro{T_{F_{n,B}}(\om)g}{v^\ast}
	= \lim_{n\to\infty} \int_{(0,t]\times U} g\scapro{F_{n,B}}{v^\ast} \d N(\om).
	$$
	This shows that
	$\scapro{F_{n,B}}{v^\ast}\to (T_B(\om))^\ast v^\ast$ weakly in $L^2_{N(\om)}((0,t]\times U)$.
	Since $\scapro{F_{n,B}}{v^\ast} \to \scapro{F_B}{v^\ast}$ pointwise, 
	a standard argument establishes that
	$\scapro{F_B}{v^\ast} = (T_B(\om))^\ast v^\ast$ $(\leb\otimes\lambda)$-almost everywhere, and hence as elements of  $L^2_{N(\om)}((0,t]\times U)$. But (by pairing with functions $g\in L^2_{N(\om)}((0,t]\times U)$)
	this is the same as saying that $T_B = T_{F_B}$.
	
	Putting things together, we have shown that, for almost all $\om\in\Om$,
	$$ \lim_{n\to\infty} T_{F_{n,B}}(\om) = T_{F_B}(\om)$$
	with convergence in $\gamma(L^2_{N(\om)}((0,t]\times U), V)$.
	The finiteness of  $\E\norm{T_{F_B}}_{\gamma(L^2_N((0,t]\times U), V)}^p$
	now follows from Fatou's lemma. This completes the proof of the implication
	\ref{it.UMD-iso-int1}$\implies$\ref{it.UMD-iso-int2}.
	
	The assertion about equivalence of norms follows by passing to the limit $n\to \infty$ in the preceding argument.
\end{proof}

We will apply this theorem to obtain a necessary and sufficient condition for a $\sigma$-finite measure on a separable UMD space to be a L\'evy measure. We start with a lemma that does not require the UMD property.

\begin{lemma}\label{lem:4.2} Let $U$ be a separable Banach space. Suppose that $\lambda$ is a $\sigma$-finite measure  on $\Borel(U)$ with $\lambda(\{0\})=0$. If the image measure $\scapro{\lambda}{u^\ast}$ is a L\'evy measure on $\R$ for all $u^\ast\in U^\ast$, then
	$G$ satisfies
	the weak $L^2$-integrability condition \eqref{eq:weakL2} for all $B\in \Borel(U)$, or equivalently, for all $u^\ast\in U^\ast$ we have
	\begin{align*}\int_{(0,1]\times B_U} |\scapro{u}{u^\ast}|^2\, N(\d s,\d u) <\infty \quad \hbox{almost surely}.
	\end{align*}
\end{lemma}

\begin{proof} Assuming without loss of generality that $\n u^\ast\n\le 1$, this follows from Theorem \ref{th.PRM},
	because
	$$ \int_{(0,1]\times B_U} |\scapro{u}{u^\ast}|^2\,\d s\,\lambda(\d u) = \int_{B_U} |\scapro{u}{u^\ast}|^2\, \lambda(\d u)
	\le \int _{[-1,1]} r^2 \,\scapro{\lambda}{u^\ast}(\d r),
	$$
	and the last expression is finite (take $H = \R$ in Theorem \ref{thm:LevyHilbert})
	since by assumption $\scapro{\lambda}{u^\ast}$ is a L\'evy measure on $\R$.
\end{proof}

It will be useful to introduce the function $G: (0,1] \times  U\to U$ defined by $$G(t,u):= u \1_{B_U}(u),$$
where $B_{U}=\{u\in U:\norm{u}\le 1\}$ as before. Note that $G$ does not depend on $t$, but we keep previously introduced notation for consistency.
We now obtain the following characterisation of L\'evy measures in the setting of UMD-spaces.

\begin{theorem}\label{th.UMD-Levy-measure}
	Let $U$ be a separable UMD-space and $\lambda$ a $\sigma$-finite measure  on $\Borel(U)$ with $\lambda(\{0\})=0$.
	Then $\lambda$ is a L\'evy measure if and only if the following  conditions are satisfied:
	\begin{enumerate}[\rm(i)]
		\item $\lambda|_r^c$ is a finite measure for all $r>0$;
		\item $\scapro{\lambda}{u^\ast}$ is a L\'evy measure on $\R$ for all $u^\ast\in U^\ast$;
		\item\label{eq.int-UMD} for some (equivalently, for all) $p\in [1,\infty)$ we have
		\begin{equation*}
			\E\norm{T_G}_{\gamma(L^2_N((0,1]\times U), U)}^p<\infty,
		\end{equation*}
		where $N$ denotes a Poisson random measure with intensity measure $\leb\otimes\lambda$
		and the operator $T_G$ is defined as in \eqref{eq:Tf}.
	\end{enumerate}
\end{theorem}

\begin{proof}
	The proof follows the lines of Theorem \ref{th.Bochner-Levy-measure}.
	Let $(\delta_k)_{k\in\N}\subseteq (0,1)$ be a sequence decreasing to $0$. Define $D_k:=\{u\in U: \delta_k < \norm{u}\le  1\}$ and the functions 
	$$G_k\colon (0,1]\times U\to U, \qquad G_k(t,u):= u \1_{D_k}(u),$$
	
	\smallskip
	{\rm `If'}: Assume first that the support of $\lambda$ is contained in $B_{U}$.
	
	By Lemma \ref{lem:4.2} and condition (ii), $G$ satisfies the weak $L^2$-integrability condition \eqref{eq:weakL2}. Condition (iii) guarantees by Theorem \ref{thm.UMD-iso-int} that the function $G$ is integrable with respect to $\cp$.  Thus, we can define the $U$-valued random variable
	\begin{align*}
		X:=\int_{(0,1]\times B_{U}} u \, \cp(\d s,\d u). 
	\end{align*}
 	Corollary \ref{co.integrable-cp} shows that the probability distribution of $X$ coincide with $\eta(\lambda)$, which shows that $\lambda$ is a L\'evy measure by its very definition.

	For the general case of a measure $\lambda$ with arbitrary support, we apply the decomposition $\lambda=\lambda|_{1}+\lambda|_{1}^c$. The measure $\lambda_{1}$ is a L\'evy measure by the first part, and $\lambda|_{1}^c$ is a L\'evy measure since it is finite. \cite[Proposition 5.4.9]{Linde} guarantees that $\lambda$ is a L\'evy measure.
	
	\smallskip
	{\rm `Only if':} Assume that $\lambda$ is a L\'evy measure. Theorem  \ref{th.Levy-equivalence}
	implies that $\lambda|_r^c$ is a finite measure for all $r>0$. This gives (i).
	It is clear that the image measures $\scapro{\lambda}{u^\ast}$ are L\'evy measures, which is (ii). To establish the integrability condition (iii), we can assume that the L\'evy measure $\lambda$ has support in $B_{U}$.
	
	Let $N$ be a Poisson random measure with intensity $\leb\otimes\lambda$.
	As in the proof of Theorem \ref{th.Bochner-Levy-measure} one sees that
	the $U$-valued random variables
	\begin{align*}
		X_k:=\int_{(0,1]\times D_k} u \, \cp(\d s,\d u)\quad\text{for } k\in\N
	\end{align*}
	converge almost surely to a random variable $X$, which must have distribution $\eta(\lambda)$ as defined in Subsection \ref{ss.PRM-Levy}.
	Since \cite[Corollary 3.3]{deAcosta-1980}
	guarantees  $\E(\norm{X}^{p})<\infty$,  \cite[Corollary 3.3]{Hoffmann-Jorgensen-sums} implies
	that $X_k\to X$ in $L^{p}(\Omega;U)$ as $k\to\infty$.
	
	We claim that from this it follows that $G$ is $L^p$-integrable with respect to $\widetilde N$ and  $$ X = \int_{(0,1]\times B_U} G\,\d\cp = \int_{(0,1]\times B_U} u \, \cp(\d s,\d u).$$
	All this follows from the arguments in the proof of Theorem \ref{thm.UMD-iso-int}: As in the proof of \ref{it.UMD-iso-int1}$\implies$\ref{it.UMD-iso-int2}, the fact that $(X_k)_{k\in\N}$ is Cauchy in $L^p(\Om;U)$ implies Cauchyness of $(T_{G_k}(\om))_{k\in\N}$ with respect to the norm of $\gamma(L^2_{N(\om)}(0,t]\times B_U,U)$ for a.a.\ $\omega\in\Omega$. The proof of \ref{it.UMD-iso-int2}$\implies$\ref{it.UMD-iso-int1} in Theorem \ref{thm.UMD-iso-int} establishes that $G$ is $L^p$-integrable with respect to $\widetilde N$ with integral $X$. Another application of the argument of  \ref{it.UMD-iso-int1}$\implies$\ref{it.UMD-iso-int2} now shows that (iii) holds.
\end{proof}

\begin{example}
	Let $U$ be the UMD-space $L^p_\mu(S)$, where $(S,{\mathcal S}, \mu)$ is a measure space and $p\in (1,\infty)$.
	By the identification of \cite[Proposition 9.3.2]{Hytonen-etal2}
	we have a natural isomorphism of Banach spaces
	$$
	\gamma(L^2_{N(\om)}((0,1]\times L^p_\mu(S)), L^p_\mu(S))\simeq L^p_\mu(S;L_{N(\omega)}^2((0,1] \times L^p_\mu(S)))
	$$
	with norm equivalence constants depending only on $p$.
	Set $$G(s,f) = f\1_{B_{{L_\mu^p}}}(f)$$ as before, and
	write $L^p_\mu:=L^p_\mu(S)$ for brevity. Reasoning formally (a rigorous version can be obtained by an additional mollification or averaging argument), it follows from Theorem \ref{thm.UMD-iso-int} (with $t=1$), Doob's inequality (applied twice), and Fubini's theorem that
	\begin{align*}
		\E\norm{T_G}_{\gamma(L^2_N((0,1]\times L_\mu^p), L^p_\mu)}^p
		& \simeq_p  \EE\sup_{0\le s\le 1}\left\|\int_{(0,s]\times B_{{L_\mu^p}}} f \, \cp(\d r,\d f)\right\|_{L_\mu^p}^p
		\\ & \simeq_p \EE\left\|\int_{(0,1]\times B_{{L_\mu^p}}} f \, \cp(\d r,\d f)\right\|_{L_\mu^p}^p
		\\ & = \EE\int_S\left|\int_{(0,1]\times B_{{L_\mu^p}}} f(\sigma) \, \cp(\d r,\d f)\right|^p\,\mu(\d\sigma)
		\\ & = \int_S\E \left|\int_{(0,1]\times B_{{L_\mu^p}}} f(\sigma) \, \cp(\d r,\d f)\right|^p\,\mu(\d\sigma)
		\\ & \simeq_p \int_S\E \sup_{0\le s\le 1}\left|\int_{(0,s]\times B_{{L_\mu^p}}} f(\sigma) \, \cp(\d r,\d f)\right|^p\,\mu(\d\sigma).
	\end{align*}
	As the expectation is for the supremum of a real-valued martingale, we can apply \cite[Theorem  3.2]{MarinelliRockner-purely}. This enables us to conclude in the case $p\in [2,\infty)$ that
	\begin{align*}
		&\EE\left[	\norm{T_G}_{\gamma(L^2_N((0,1]\times L^p_\mu), L^p_\mu)}^p\right]\\
		&\qquad \simeq_p \int_S \left( \left( \int_{(0,1]\times B_{{L_\mu^p}}} \abs{f(\sigma)}^2\, \d r\,\lambda(\d f)\right)^{p/2}
		+ \int_{(0,1]\times B_{{L_\mu^p}}} \abs{f(\sigma)}^p \, \d r\,\lambda(\d f)\right)\, \mu(\d \sigma)\\
		&\qquad = \int_S \left(  \int_{ B_{{L_\mu^p}}} \abs{f(\sigma)}^2\, \lambda(\d f)\right)^{p/2}\,\mu(\d \sigma)
		+ \int_{ B_{{L_\mu^p}}} \norm{f}^p_{L_\mu^p} \, \lambda(\d f).
	\end{align*}
	Thus, we obtain the same characterisation of a L\'evy measure on $L^p_\mu$ for $p\in [2,\infty)$ as in Theorem \ref{th.Bochner-Levy-measure}.
	
	In the case $p\in (1,2]$,  we obtain by \cite[Theorem  3.2 and page 5]{MarinelliRockner-purely}, that
	\begin{align*}
		& \EE\left[	\norm{T_G}_{\gamma(L^2_N((0,1]\times L^p_\mu), L^p_\mu)}^p\right]
		\\ &\qquad \simeq_p \int_S \inf \left\{ \left( \int_{B_{{L_\mu^p}}} \abs{g_{1,\sigma}(f)}^2\, \lambda(\d f)\right)^{p/2}
		+ \int_{ B_{{L_\mu^p}}} \abs{g_{2,\sigma}(f)}^p \, \d r\,\lambda(\d f)\right\}\, \mu(\d \sigma),
	\end{align*}
	where the infimum is taken over all functions $g_{1,\sigma}\in L^2_\lambda(B_{{L_\mu^p}}) $ and $g_{2,\sigma}\in L^p_\lambda(B_{{L_\mu^p}})$  with $f(\sigma)=g_{1,\sigma}(f)+g_{2,\sigma}(f)$ for all $f\in B_{{L_\mu^p}}$ and $\sigma\in S$.
	The expression on right-hand side is subtly different from the corresponding expression in Theorem \ref{th.Bochner-Levy-measure}. However, the present derivation, combined with Theorem \ref{th.Bochner-Levy-measure}, establish the equivalence of these expressions.
\end{example}

\begin{remark}
	Theorem \ref{thm.UMD-iso-int} continues to hold if the UMD property on $V$ is weakened to reflexivity with finite cotype, by making the following adjustments. First of all, a version of \cite[Theorem 2.1]{Yaroslavtsev-20}
	for $V$-valued martingales with independent increments is obtained in \cite[Proposition  6.7]{Yaroslavtsev-20} for Banach spaces $V$ with finite cotype. Using this result, the proof of \cite[Theorem 5.1]{Yaroslavtsev-20} can be repeated, resulting in a version of this theorem for reflexive space with finite cotype; see \cite[Proposition 6.8]{Yaroslavtsev-20}. Reflexivity enters in view of the results in \cite[Section 3]{Yaroslavtsev-20} that are still needed in their stated forms. Our Theorem \ref{th.UMD-Levy-measure} can be extended accordingly. We thank Ivan Yaroslavstev and Gergely B\'odo for kindly pointing this out to us. Finally we thank the anonymous referee for many detailed comments.
\end{remark}

\section{Outlook}
Similarly as in finite dimensions, infinitely divisible measures on a Banach space $U$ are characterised by  triplets $(a,Q,\lambda)$ where $a\in U$, $Q\colon U^\ast\to U$ is a nonnegative, symmetric trace class operator and $\lambda$ is a L\'evy measure on $\Borel(U)$. For weak convergence of a sequence $(\mu_n)_{n\in\N}$ of infinitely divisible measures with characteristics $(a_n,Q_n, \lambda_n)$ necessary conditions are known in Banach spaces, but in general they are not sufficient; see \cite[Prop.\ 5.7.4]{Linde}. Only in separable Hilbert spaces, necessary conditions are known, which are established in \cite[Theorem  5.5]{Parthasarathy}. In fact, as pointed out in \cite{Linde}, necessary conditions in Banach spaces would have allowed for an explicit characterisation for L\'evy measures. As we have now such a characterisation, our result should enable the derivation for necessary conditions for the weak convergence of a sequence of infinitely divisible measures on $L^p$-spaces or in UMD-spaces.

In the current work, using the $L^p$-estimates for simple functions in \cite{Dirksen, Yaroslavtsev-20}, we have already introduced a description of the largest space of vector-valued deterministic functions integrable with respect to a compensated Poisson random measure in either $L^p$-spaces or UMD-spaces; see Lemma \ref{lem:simple-Iq-dense} and Theorem \ref{thm.UMD-iso-int}. Such a description of the space of deterministic integrands can be used to derive the existence of a stochastic integral for random vector-valued integrands with respect to a compensated Poisson random measure, similarly as in \cite{Dirksen-Maas-Neerven}. Since the compensated Poisson random measure has independent increments, the decoupled tangent sequence can be constructed, and thus the decoupling inequalities in UMD-spaces enables to derive the existence of the stochastic integral.\\[.5cm]

{\bf Acknowledgement.} The authors would like to thank Gergely B\'odo for proofreading an earlier version of this  article and providing helpful comments, and Ivan Yaroslavtsev for helpful suggestions. After the completion of this paper, it was kindly pointed out by Sjoerd Dirksen that Theorem \ref{th.Bochner-Levy-measure} had been obtained independently in an unfinished preprint with Carlo Marinelli as early as 2016, where it is pointed out that the case $p \in [2,\infty)$ was already obtained in \cite{GMZ79}.

\bibliography{biblio-Levy-measures}
\bibliographystyle{plain}

\end{document}